\documentclass[12pt,leqno,twoside]{article}
\usepackage{amssymb}
\usepackage{amsmath}
\usepackage{amsthm}
\usepackage{t1enc}
\usepackage[cp1250]{inputenc}
\usepackage{a4,indentfirst,latexsym}
\usepackage{graphics}
\usepackage{mathrsfs}
\usepackage{cite,enumitem,graphicx}
\usepackage[colorlinks=true,urlcolor=black,
citecolor=black,linkcolor=black,linktocpage,pdfpagelabels,
bookmarksnumbered,bookmarksopen]{hyperref}
\usepackage[colorinlistoftodos]{todonotes}
\usepackage{pifont}

\input xy
\xyoption{all}

\newtheorem{Th}{Theorem}[section]

\newtheorem{Lem}[Th]{Lemma}

\makeatletter
    
    \newcommand{\Rmnum}[1]{\expandafter\@slowromancap\romannumeral #1@}

\newenvironment{altproof}[1]
{\noindent
{\em Proof of {#1}}.}
{\nopagebreak\mbox{}\hfill $\Box$\par\addvspace{0.5cm}}

   \newcommand{\vp}{\varphi}
   
   \newcommand{\eps}{\varepsilon}

   \def\supp{\mathrm{supp}}

   \def\Z{\mathbb{Z}}

   \def\N{\mathbb{N}}
   \def\R{\mathbb{R}}

   \def\J{\mathcal{J}}

   \def\D{\mathcal{D}}

\newcommand{\cC}{{\mathcal C}}

\newcommand{\cN}{{\mathcal N}}

\newcommand{\cT}{{\mathcal T}}

\newcommand{\Om}{\Omega}

\newcommand{\weakto}{\rightharpoonup}

\newcommand{\cTto}{\stackrel{\cT}{\longrightarrow}}

\numberwithin{equation}{section}

\begin{document}
\title{Ground states of nonlinear Schr\"odinger equations with sum of periodic and inverse-square potentials}
\author{Qianqiao Guo \and Jaros\l aw Mederski}
\date{}
\maketitle

\begin{abstract}
We study the existence of solutions of the following nonlinear Schr\"odinger equation
\begin{equation*}
    -\Delta u + \Big(V(x)-\frac{\mu}{|x|^2}\Big) u = f(x,u)
    \hbox{ for } x\in\R^N\setminus\{0\},
\end{equation*}
where $V:\R^N\to\R$ and $f:\R^N\times\R\to\R$ are periodic in $x\in\R^N$. We assume that $0$ does not lie in the spectrum of $-\Delta+V$ and $\mu<\frac{(N-2)^2}{4}$, $N\geq 3$. The superlinear and subcritical term $f$ satisfies a weak monotonicity condition. For sufficiently small $\mu\geq 0$ we find a ground state solution as a minimizer of the energy functional on a natural constraint. If $\mu<0$ and $0$ lies below the spectrum of $-\Delta+V$, then ground state solutions do not exist.
\end{abstract}



\section*{Introduction}
\setcounter{section}{1}

The paper deals with the  Schr\"odinger equation
\begin{equation}\label{NSE}
    -\Delta u + \Big(V(x)-\frac{\mu}{|x|^2}\Big) u = f(x,u)
    \hbox{ for } x\in\R^N\setminus\{0\}.
\end{equation}
where $V:\R^N\to\R$ is a periodic potential, $\mu<\overline{\mu}:=\frac{(N-2)^2}{4}$, $N\geq 3$, and $f:\R^N\times\R\to\R$ has superlinear and subcritical growth. The Schr\"odinger equation appears in many physical models, for instance in nonlinear optics, where propagation of light through periodic optical structures  with localized singular defect potentials has been intensively studied (see \cite{BonannoGhimentiSquassina, HolmesWeinstein} and references therein). We focus on the localized inverse-square potential $-\frac{\mu}{|x|^2}$ which arises in many other areas such as quantum mechanics, nuclear physics, molecular physics and quantum cosmology \cite{FelliMarchiniTerracini2007,FrankLandSpector,FelliTerracini2006}.

In this paper we study the nonlinearity $f$ having superlinear and subcritical growth (see (F2)-(F4) below), for instance in dimension $N=3$ in the study of self-focusing Kerr-like optical media one has  $f(x,u)=\Gamma(x)|u|^2u$ with $\Gamma\in L^{\infty}(\R^3)$ periodic, positive and bounded away from $0$ (see \cite{Pankov,Kuchment}). Nonlinear Schr\"odinger equations in $\R^N$ with the inverse-square potentials have recently attracted a considerable attention in the mathematical literature,  for example \cite{FelliJAM2009,FelliMarchiniTerracini2007,FelliMarchiniTerracini2009,FelliTerracini2006,
TerraciniADE,Smets,ChabrowskiSzulkinWillem,RuizWillem2003,DengJinPengJDE2012} and all of these works concentrate on critically growing nonlinearites having a component of the form $f(x,u)=\Gamma(x)|u|^{2^*-2}u$, where $2^*:=\frac{2N}{N-2}$ is the Sobolev exponent. This is well-justified since for $V=0, 0\le\mu<\overline{\mu}$ and $f(x,u)=|u|^{p-2}u$ with $p\neq2^*$, Terracini \cite{TerraciniADE}[
Theorem 0.1] has shown that (\ref{NSE}) admits only trivial solution 
$u=0$ in 
$\D^{1,2}(\mathbb{R}^N)\cap L^{p}(\mathbb{R}^N)$ and nontrivial solutions appear in the critical case $p=2^*$. 
Our aim is to deal with the existence of nontrivial soultions in the subcritical case. Therefore we must ensure that $V\neq 0$ and our principal assumption is that $0$ does not belong to the spectrum $\sigma(-\Delta+V)$ of the Schr\"odinger operator $-\Delta+V$. Since $\sigma(-\Delta+V)$ is bounded from below and consists of sum of disjoint closed intervals, then either $0$ lies below the spectum or $0$ lies in a finite spectral gap of the Schr\"odinger operator. In the both cases solutions to \eqref{NSE} represent the so-called standing gap solitons \cite{Pankov,Kuchment}.

Recall that if $\mu=0$ there is a broad literature treating the Schr\"odinger equations with periodic potentials $V$ in the subcritical case, see for example \cite{ZelatiRabinowitz,AlamaLi,TroestlerWillem,KryszSzulkin, Pankov,Liu,SzulkinWeth,MederskiTMNA2014,MederskiNehari2014} and references therein.
If $0\leq \mu\leq \overline{\mu}$ we may consider the following bilinear form
\begin{equation}\label{eq:DefOfB}
B_{\mu}(u,v):=\int_{\R^N}\langle \nabla u,\nabla v\rangle+\Big(V(x)-\frac{\mu}{|x|^2}\Big)uv\,dx 
\end{equation}
and to the best of our knowledge, all existing approaches to the problem \eqref{NSE} with $\mu>0$ require the positive definiteness of $B_\mu$ in a suitable function space, where solutions are looked for.
However, if $\mu\neq 0$, $V$ is periodic and nonconstant, possibly sign-changing, then $B_\mu$ may be strongly indefinite and requires more delicate approach. Moreover the singular potential $-\frac{\mu}{|x|^2}$ does 
not belong to the Kato's class \cite{ReedSimon} and cannot be treated as a lower order perturbation term of $-\Delta+V$.

In what follows, throughout the paper we assume the following conditions:

\begin{itemize}
 \item[(V)] $V\in L^{\infty}(\R^N)$, $V$ is $\Z^N$-periodic in $x\in\R^N$ and
$0\notin\sigma(-\Delta+V)$. Here $V$ is $\Z^N$-periodic in $x\in\R^N$ means $V(x+y)=V(x)$ for any $y\in \Z^N$.
\item[(F1)] $f:\R^N\times\R\to\R$ is measurable in $x\in\R^N$ and continuous in $u\in\R$ for a.e. $x\in\R^N$. Moreover $f$ is $\Z^N$-periodic in $x$.
\item[(F2)] There are $a>0$ and $2<p<2^*=\frac{2N}{N-2}$ such that
$$|f(x,u)|\leq a(1+|u|^{p-1})\hbox{ for all }u \in\R,\; x\in\R^N.$$

\item[(F3)] $f(x,u)=o(u)$ uniformly with respect to $x$ as $|u|\to0$.
\item[(F4)] $F(x,u)/u^2\to\infty$ uniformly in $x$ as $|u|\to\infty$, where $F$ is the primitive of $f$ with respect to $u$, that is $F(x,u)=\int_0^u f(x,t)dt$.
\item[(F5)]  $u\mapsto f(x,u)/|u|$ is non-decreasing on $(-\infty,0)$ and $(0,+\infty)$.
\end{itemize}

The energy functional
$\J:H^1(\R^N)\to\R$ given by
\begin{equation}\label{eqJFormula}
\J_{\mu}(u):=\frac{1}{2}\int_{\R^N}|\nabla u|^2+\Big(V(x)-\frac{\mu}{|x|^2}\Big)|u|^2\,dx-\int_{\R^N}F(x,u)\,dx
\end{equation}
is of $\cC^1$-class and its critical points correspond to solutions of \eqref{NSE}. In view of (V) spectral theory asserts that $\sigma(-\Delta+V)$ is purely continuous, bounded from below and consists of closed disjoint intervals \cite{ReedSimon}. Moreover
there is an orthogonal decomposition of $X:=H^1(\R^N)=X^+\oplus X^-$, such that $B_{0}(\cdot,\cdot)$ given by \eqref{eq:DefOfB} with $\mu=0$ is positive definite on $X^+$ and negative definite on $X^-$. If $0$ lies in a finite spectral gap, then both $X^+$ and $X^-$ are infinite dimensional and the problem is strongly indefinite. For any $u\in X$ we denote by $u^+ \in X^+$ and $u^- \in X^-$ the corresponding summands so that $u = u^++u^-$.


We introduce the following constants
\begin{eqnarray*}
&&\hspace{-5mm}\mu(V)^+:=\sup\Big\{M>0:\;\int_{\R^N}|\nabla u|^2+V(x)|u|^2\,dx\geq M\int_{\R^N}|\nabla u|^2\, dx
\hbox{ for all } u\in X^+\Big\},\\
&&\hspace{-5mm}\mu(V)^-:=\sup\Big\{M>0:\;-\int_{\R^N}|\nabla u|^2+V(x)|u|^2\,dx\geq M\int_{\R^N}|\nabla u|^2\, dx
\hbox{ for all } u\in X^-\Big\},
\end{eqnarray*}
which play a crucial role in case $\mu\neq 0$ and \eqref{eq:DefOfB} is indefinite.
We show that $\mu(V)^+\in (0,1]$, $\mu(V)^->0$ and clearly if $V(x)\geq 0$ for a.e. $x\in\R^N$, then $\mu(V)^+=1$, $X^-=\{0\}$ and we may take $\mu(V)^-=\infty$ (see Lemma \ref{LemMu}).

Our first main result reads as follows.

\begin{Th}\label{ThMain1}
If $0\leq \mu < \frac{(N-2)^2}{4}\mu(V)^+$, then (\ref{NSE}) has a ground state solution.
\end{Th}

\noindent By a {\em ground state} solution, we mean a nontrivial critical point $u$ of $\J_\mu$
such that 
$$\J_{\mu}(u)=\inf_{\cN_\mu}\J_{\mu}>0,$$ 
where 
$$\cN_{\mu}:=\{u\in X\setminus X^-:\;\J'_{\mu}(u)(u)=0\hbox{ and }\J'_{\mu}(u)(v)=0\hbox{ for any }v\in X^-\}$$ stands for the Nehari-Pankov manifold introduced in \cite{Pankov} and generalized approaches can be found in \cite{SzulkinWeth,BartschMederskiARMA2014,MederskiNehari2014}. The set $\cN_{\mu}$ is a natural constraint and since it contains all nontrivial critical points, then any ground state solution is a nontrivial critical point with the least possible energy $\J_{\mu}$. 

The existence of ground states of \eqref{NSE} without the Hardy term ($\mu=0$) has been recently obtained by Szulkin and Weth \cite{SzulkinWeth} under more restrictive assumptions, in particular they assumed that $u\mapsto f(x,u)/|u|$ is strictly increasing on $(-\infty,0)$ and $(0,+\infty)$. Moreover Liu in \cite{Liu} has assumed the weak monotonicity condition (F5) and a least energy solution to \eqref{NSE} has been obtained. The existence of ground states in case $\mu=0$ and under assumptions (V), (F1)-(F5) follows from \cite{MederskiNehari2014}[Theorem 1.1].
Below we provide comparison of ground states levels of $\J_\mu$ and $\J_0$. Moreover we obtain a nonexistence result in case $X^-=\{0\}$, $\mu<0$. The behaviour of ground state solutions in the limit $\mu\to 0^+$ requires the following additional technical condition. 

\begin{itemize}
\item[(F6)] $f(x,\cdot)$ is of $\cC^1$ class for a.e. $x\in\R^N$ and there are $b>0$ and $2<q\leq p$ such that 
$$f(x,u)u-2F(x,u)\geq b|u|^q\hbox{ for all }u\in\R,\;x\in\R^N,$$
where $p$ is the same one appearing in (F2).
\end{itemize}

\begin{Th}\label{ThMain2} 
Suppose that $-\frac{4}{(N-2)^2}\mu(V)^-<\mu <\frac{(N-2)^2}{4}\mu(V)^+$,  $u_\mu$ is a ground state of $\J_\mu$ and $u_0$ is a ground state of $\J_0$.\\ 
(a) Then there are $t>0$ and $v\in X^-$ such that $tu_{\mu}+v\in\cN_0$ and
\begin{equation}\label{ConditonOnmu2}
\inf_{\cN_0}\J_{0}\leq \inf_{\cN_\mu}\J_{\mu}+\frac12\int_{\R^N}\frac{\mu}{|x|^2}|tu_{\mu}+v|^2\,dx.
\end{equation}
If $\mu\geq 0$, then there are $t>0$ and $v\in X^-$ such that $tu_{0}+v\in\cN_\mu$ and 
\begin{equation}\label{ConditonOnmu1}
\inf_{\cN_\mu}\J_{\mu}\leq \inf_{\cN_0}\J_{0}-\frac12\int_{\R^N}\frac{\mu}{|x|^2}|tu_{0}+v|^2\,dx.
\end{equation}
Moreover
\begin{equation}\label{ThMain2conv}
\lim_{\mu\to 0^+}\inf_{\cN_\mu}\J_{\mu}=\inf_{\cN_0}\J_{0}.
\end{equation}
(b) If in addition (F6) holds and $\mu_n\to 0^+$, then there is a sequence $(x_n)\subset\Z^N$ such that $u_{\mu_n}(\cdot+x_n)$ tends to a ground state $u$ of $\J_0$ in the strong topology of $X$ as $n\to\infty$.\\
(c) If $\mu<0$ and $0<\inf\sigma(-\Delta +V)$, then $\J_\mu$ has no ground states.
\end{Th}

The problem of the existence of ground states in case $\mu<0$ and $\inf\sigma(-\Delta +V)<0$ remains open.

The paper is organized as follows. In the next section we formulate our problem in a variational setting and we investigate the properties of the Nehari-Pankov manifold on which we minimize $\J_\mu$ to find a ground state. Since $\cN_\mu$ need not to be of class $\cC^1$, we are not able to apply the standard minimizing method of $\J_\mu$ on  $\cN_\mu$. The method of Szulkin and Weth \cite{SzulkinWeth} fails as well due to the weak monotonicity condition (F5). Our approach uses a linking argument and we find a Cerami sequence by means of the critical point theory developed in\cite{MederskiNehari2014}[Section 2]. Next, in Section \ref{sec:ProofTh1} we prove Theorem \ref{ThMain1} and in Section \ref{sec:ProofTh2} we prove Theorem \ref{ThMain2}. The lack of compactness of Cerami sequences requires decompositions of these sequences  which is provided in Lemma \ref{LemStruwe} and proved in the last Section \ref{sec:Decomposition}.

\section{Variational setting}

In view of the spectral theory \cite{ReedSimon} we may introduce a new inner product in $X=H^1(\R^N)$ by the following formula
$$\langle u,v\rangle :=\int_{\R^N}\langle \nabla u^+,\nabla v^+\rangle+V(x)\langle u^+,v^+\rangle\;dx-
\int_{\R^N}\langle \nabla u^-,\nabla v^-\rangle+V(x)\langle u^-,v^-\rangle\;dx$$
and a norm given by
$$\|u\|^2:=\langle u,u\rangle,$$
which is equivalent with the usual Sobolev norm in $H^1(\R^N)$, that is 
\begin{equation*}
\|u\|^2_{H^1(\R^N)}=\int_{\R^N}|\nabla u|^2+u^2\,dx.
\end{equation*}  
Then
$X^+$ and $X^-$ are orthogonal with respect to the inner product $\langle\cdot,\cdot\rangle$ as well. If $\inf\sigma(-\Delta+V)>0$ we have $X^-=\{0\}$ and 
$$\|u\|^2 =\int_{\R^N}|\nabla u|^2+V(x)u^2\,dx.$$

\begin{Lem}\label{LemMu}
$\mu(V)^+\in (0,1]$ and $\mu(V)^->0$.
\end{Lem}
\begin{proof}
Since $B_0(\cdot,\cdot)$ is positive definite on $X^+$, then
$$B_0(u,u)\geq C \|u\|^2_{H^1}\geq C \int_{\R^N}|\nabla u|^2\,dx$$
for some constant $C>0$ and all $u\in X^+$. Thus $\mu(V)^+\geq C>0$. Similarly, $B_0(\cdot,\cdot)$ is negative definite on $X^-$ and
$\mu(V)^->0$.
Now suppose that $\mu(V)^+>1$, take any open and bounded $\Om\subset\R^N$ such that $\sup_{x\in\Om}V(x)=V_0>0$. 
Let us take any $u\in X$ such that $\supp(u)\subset \Om$ and
\begin{eqnarray}\label{eq:proofmu}
\int_{\Om}|\nabla u|^2\,dx<V_0(\mu(V)^+-1)^{-1}\int_{\Om}|u|^2\,dx.
\end{eqnarray}
Observe that
\begin{eqnarray*}
V_0\int_{\Om}|u|^2\,dx\geq\int_{\R^N}V(x)|u|^2\,dx
\geq (\mu(V)^+-1)\int_{\R^N}|\nabla u|^2\,dx =(\mu(V)^+-1)\int_{\Om}|\nabla u|^2\,dx
\end{eqnarray*}
and we obtain a contradiction with \eqref{eq:proofmu}.
\end{proof}

In addition to the norm topology $\|\cdot\|$ we need the topology $\cT$ on $X$ which is induced by the norm
$$\|u\|_\cT:=\max\Big\{\|u^+\|, \sum_{k=1}^\infty\frac{1}{2^{k+1}}|\langle u^-,e_k\rangle|\Big\},$$
where $(e_k)$ stands for a total orthonormal sequence in $X^-$. Recall that \cite{KryszSzulkin,Willem,BartschDing,MederskiNehari2014} 
$$\|u^-\|\leq\|u\|_\cT\leq \|u\|\hbox{ for }u\in X$$
and on bounded subsets of $X$ the topology $\cT$ coincides with
the
product of the norm topology in $X^+$ and the weak topology in $X^-$. The convergence of a sequence in $\cT$ topology will be denoted by $u_n\cTto u$.

\indent Assumptions (V), (F1) and (F2) allow to consider the energy functional $\J_\mu$ associated to (\ref{NSE})
which is a well-defined $\cC^1$-map. Moreover critical points of $\J_\mu$ correspond to  solutions to (\ref{NSE}). Observe that (F3) and (F5) imply that
\begin{equation}\label{ARfor2}
f(x,u)u\geq 2F(x,u)\geq 0\hbox{ for }x\in\R^N,\;u\in \R. 
\end{equation}
Hence, if $t\in\R$, $u_n\cTto u_0$ and $\J_\mu(u_n)\geq t$ for any $n\geq 1$, then $t\leq \frac{1}{2}\big(\|u_n^+\|^2-\|u_n^-\|^2\big)$ and $\|u_n\|$ is bounded. Therefore for $\mu\geq 0$ one easily verifies  
the following conditions:
\begin{itemize}
\item[(A1)] $\J_\mu$ is $\cT$-upper semicontinuous, that is $\J_\mu^{-1}([t,\infty))$ is $\cT$-closed for any $t\in\R$. 

\item[(A2)] $\J'_\mu$ is $\cT$-to-weak$^*$ continuous in $\J_\mu^{-1}([0,\infty))$, that is $\J'_\mu(u_n)\weakto\J'_\mu(u_0)$ provided that $u_n\cTto u_0$ and $\J_{\mu}(u_n)\geq 0$ for $n\geq 0$.
\end{itemize}

In Lemma \ref{LemGeometry} below  we will check the geometrical conditions (A3) and (A4) provided that $0\leq \mu <\frac{(N-2)^2}{4}\mu(V)^+$.

\begin{itemize}
\item[(A3)] There exists $r>0$ such that $m:=\inf\limits_{u\in X^+:\|u\|=r} \J_\mu(u)>0$.
\item[(A4)] For every $u\in X\setminus X^-$ there exists $R(u)>r$ such that
$$\sup_{\partial M(u)}\J_\mu\leq \J_\mu(0)=0,$$
where
\begin{equation}\label{DefOfM}
M(u):=\{t u+v \in X |\; v \in X^-,\;
\|tu+v\|\leq R(u),\; t\geq0 \}.
\end{equation}
\end{itemize}
We also require the following condition which is implied by Lemma \ref{Lem_IneqForJ} below for $\mu> -\frac{4}{(N-2)^2}\mu(V)^-$.
\begin{itemize}
\item[(A5)] If $u\in \cN_{\mu}$ then $\J_{\mu}(u)\geq \J_{\mu}(tu+v)$ for $t\geq 0$ and $v\in X^-$.
\end{itemize}

Finally we intend to apply the following linking theorem  obtained by the second author in \cite{MederskiNehari2014} (cf. \cite{BartschDing,LiSzulkin,WillemZou}).

\begin{Th}\label{ThAbstract}
If $\J_{\mu}\in\cC^1(X,\R)$ satisfies (A1)-(A4), then there exists a Cerami sequence $(u_n)$ at level $c_{\mu}$, that is
$\J_\mu(u_n)\to c_\mu$ and $(1+\|u_n\|)\J'_\mu(u_n)\to 0$,
such that
$$0<m\leq c_\mu\leq \inf_{\cN_{\mu}}\J_{\mu}$$
and
$$
c_\mu:=\inf_{u\in X\setminus X^-}\inf_{h\in\Gamma(u)} \sup_{u'\in M(u)} \J_\mu(h(u',1)),$$
where
$\Gamma(u)$ consists of $h\in \cC(M(u)\times [0,1])$ satisfying the following conditions
\begin{itemize}
\item[(h1)] $h$ is $\cT$ -continuous (with respect to norm $\|\cdot\|_\cT$);
\item[(h2)] $h(u,0)=u$ for all $u\in M(u)$;
\item[(h3)] $\J_\mu(u)\geq \J_\mu(h(u,t))$ for all $(u,t)\in M(u)\times [0,1]$;
\item[(h4)] each $(u,t)\in M(u)\times [0,1]$ has an open neighborhood $W$ in the product topology of $(X,\cT)$ and $[0,1]$ such that the set $\{v-h(v,s):(v,s)\in W\cap(M(u)\times [0,1])\}$ is contained in a finite-dimensional subspace of $X$. 
\end{itemize}
Moreover $c_\mu=\inf_\cN\J_\mu$ provided that $c_\mu\geq\J_\mu(u)$ for some critical point $u\in X\setminus X^-$ and (A5) additionally holds.
\end{Th}

In the last part of this section we check assumptions (A3)-(A5). 
Observe that in view of the Hardy inequality
\begin{equation}\label{Hardyineq1}
\int_{\R^N}|\nabla u|^2\;dx\geq \frac{(N-2)^2}{4}\int_{\R^N}\frac{|u|^2}{|x|^2}\;dx
\end{equation}
for any $u\in X$. Hence 
\begin{equation}\label{Hardyineq}
\int_{\R^N}|\nabla u|^2-\mu\frac{|u|^2}{|x|^2}\;dx\geq \Big(1-\frac{4\mu}{(N-2)^2}\Big)\int_{\R^N}|\nabla u|^2\;dx,
\end{equation}
where $1-\frac{4\mu}{(N-2)^2}>0$.

The following lemma is standard and follows from (F1) - (F3).
\begin{Lem}\label{eqEstOfF} 
For any $\eps>0$ there is $C_\eps>0$ such that for any $u\in H^1(\R^N)$
\begin{equation}\label{EstimF}
\int_{\R^N}F(x,u)\; dx\leq \eps|u|_2^2+C_\eps|u|_p^p,
\end{equation}
where $|\cdot|_k$ stands for the norm in $L^k(\R^N)$ for any $k\geq 1$.
\end{Lem}

\begin{Lem}
Let $0\le \mu <\frac{(N-2)^2}{4}\mu(V)^+$. Then for any $u\in X^+$
\begin{equation}
\|u\|_{\mu}:=\Big(\int_{\R^N}|\nabla u|^2+\Big(V(x)-\frac{\mu}{|x|^2}\Big)|u|^2\; dx\Big)^{1/2}
\end{equation}
satisfies inequalities
\begin{equation}\label{eqNormMuEst}
\|u\|^2\geq \|u\|^2_{\mu}\geq \frac{1}{2}\Big(\mu(V)^+ -\frac{4\mu}{(N-2)^2}\Big) \|u\|^2.
\end{equation}
Hence $\|\cdot\|_{\mu}$ is a norm in $X^+$ equivalent with $\|\cdot\|$.
\end{Lem}
\begin{proof}
Observe that by (\ref{Hardyineq1})
\begin{eqnarray*}
\int_{\R^N}|\nabla u|^2+\Big(V(x)-\frac{\mu}{|x|^2}\Big)|u|^2\; dx&\geq& \int_{\R^N}\mu(V)^+|\nabla u|^2-\frac{\mu}{|x|^2}|u|^2\;dx\\
&\geq& \Big(\mu(V)^+-\frac{4\mu}{(N-2)^2}\Big)\int_{\R^N}|\nabla u|^2\;dx
\end{eqnarray*}
for any $u\in X^+$. On the other hand by (\ref{Hardyineq})
\begin{eqnarray*}
\int_{\R^N}|\nabla u|^2+\Big(V(x)-\frac{\mu}{|x|^2}\Big)|u|^2\; dx
\geq \int_{\R^N} V(x)|u|^2\; dx
\end{eqnarray*}
and by Lemma \ref{LemMu} we get (\ref{eqNormMuEst}).
\end{proof}

Now we show that $\J_{\mu}$ satisfies (A3) and (A4).
\begin{Lem}\label{LemGeometry}
Let $0\leq \mu <\frac{(N-2)^2}{4}\mu(V)^+$. For any $u_0\in X\setminus X^-$ there are $R>r>0$ such that
\begin{equation}\label{eqLinkingGeo}
\inf_{u\in X^+:\;\|u\|=r}\J_{\mu}(u)>\J_{\mu}(0)=0\ge\sup_{\partial M(u_0)}\J_{\mu}.
\end{equation}
\end{Lem}
\begin{proof}
Observe that by (\ref{eqNormMuEst}) and Lemma \ref{eqEstOfF} we get
\begin{eqnarray*}
\J_{\mu}(u^+)&\geq&
\frac{1}{4}\Big(\mu(V)^+-\frac{4\mu}{(N-2)^2}\Big)\|u^+\|^2 -\int_{\R^N}F(x,u^+)\;dx\\
&\geq& \frac{1}{4}\Big(\mu(V)^+-\frac{4\mu}{(N-2)^2}\Big)\|u^+\|^2-\eps|u^+|_2^2-C_\eps|u^+|_p^p
\end{eqnarray*}
for any $\eps>0$. Taking $\eps>0$ small enough, it is easy to see that there exists $r>0$ small enough such that
$$\inf_{u\in X^+:\;\|u\|=r}\J_{\mu}(u)>\J_{\mu}(0)=0.$$
Let $u_0\in X\setminus X^-$ and $u=tu_0+u^-$, $u^-\in X^-$, $t\geq 0$. Since
\begin{eqnarray*}
\J_{\mu}(u)&\leq&\J_{0}(u)=\frac{1}{2}\int_{\R^N}|\nabla u|^2+V(x)|u|^2\;dx-\int_{\R^N}F(x,u)\;dx,
\end{eqnarray*}
then similarly as in \cite{Liu,SzulkinWeth}
we obtain that
$\J_{\mu}(u)\to -\infty$ if $\|u\|\to \infty$.
Moreover
$\J_{\mu}(u^-)\leq 0$ and we find sufficiently large $R>r$ such that
$$\sup_{\partial M(u_0)}\J_{\mu}\le 0.$$
\end{proof}

\begin{Lem}\label{Lem_IneqForJ}
If $\mu> -\frac{4}{(N-2)^2}\mu(V)^-$, $u\in X\setminus X^-$, $v\in X^-$ and $t\geq 0$, then 
\begin{equation}\label{ineqForJ}
\J_\mu(u)\geq \J_{\mu}(tu+v)-\J_\mu'(u)\Big(\frac{t^2-1}{2}u+tv\Big).                                                                
\end{equation}
In particular (A5) holds.
\end{Lem}
\begin{proof}
Observe that 
\begin{eqnarray*}
\J_\mu(tu+v)-\J_\mu(u)-\J_\mu'(u)\Big(\frac{t^2-1}{2}u+tv\Big) = -\frac{1}{2}\|v\|^2-\frac{1}{2}\int_{\R^N}\frac{\mu}{|x|^2}|v|^2\,dx+\int_{\R^N}\vp(t,x)\, dx
\\ 
\leq-\frac{1}{2}\|v\|^2+\frac{1}{2}\max\Big\{0,-\mu\frac{(N-2)^2}{4}\Big\}\int_{\R^N}|\nabla v|^2\,dx+\int_{\R^N}\vp(t,x)\, dx,
\end{eqnarray*}
where
$$\vp(t,x):=f(x,u)\Big(\frac{t^2-1}{2}u+tv\Big)
  + F(u) - F(tu+v).$$
Suppose that $u(x)\neq 0$. Similarly as in \cite{SzulkinWeth,MederskiNehari2014} we show that $\vp(t,x)\leq 0$.  
Indeed, in view of \eqref{ARfor2} we get
$\vp(0,x)\leq 0$. By (F4), we obtain $\vp(t,x)\to-\infty$ as $t\to\infty$. Let $t_0\geq 0$ be such that
$$\vp(t_0,x)=\max_{t\geq 0}\vp(t,x).$$
We may assume that $t_0>0$ and thus $\partial_t\vp(t_0,x)=0$. Thus
$$f(x,u)(t_0u+v)=f(x,t_0u+v)u.$$
If $t_0u+v=0$ or $t_0u+v\neq 0$, then by \eqref{ARfor2} we obtain
$\vp(t_0,x)\leq 0$. Suppose that $u< t_0u+v$. Then, by (F5) a function 
$$(0,+\infty)\ni s\mapsto \frac{f(x,s)}{s}\in \R$$
is constant on $(u, t_0u+v)$ and
$$F(x,u)-F(x,t_0u+v)= f(x,u)\frac{u^2-(t_0u+v)^2}{2u}.$$
Thus
\begin{eqnarray*}
\vp(t_0,x)&= &f(x,u)\Big(\frac{t_0^2-1}{2}u+t_0v\Big)
  + f(x,u)\frac{u^2-(t_0u+v)^2}{2u}\\
&=&-\frac{f(x,u)uv^2}{2u^2}\leq 0.
\end{eqnarray*}
Similarly we check the case $u>t_0u+v$.
Therefore $\vp(t,x)\leq 0$ for any $t\geq 0$ and 
\eqref{ineqForJ} holds.
\end{proof}

\section{Proof of Theorem \ref{ThMain1}}\label{sec:ProofTh1}

We need the following lemma.
\begin{Lem}\label{LemIntegralConvergence}
If $|x_n|\to\infty$, then for any $u\in H^1(\R^N)$,
$$\int_{\R^N}\frac{1}{|x|^2}|u(\cdot-x_n)|^2\;dx\to 0\hbox{ as } 
n\to\infty.$$
\end{Lem}
\begin{proof}
Let $\vp_m\in C^{\infty}_0(\R^N)$ and $\vp_m\to u$ in $H^1(\R^N)$ as $m\to\infty$. Take $R_m>0$ such that $\supp(\vp_m)\subset B(0,R_m)$. For any $m$ we find $n=n(m)$ such that $|x_n|-R_m\geq m$ and $(n(m))$ is an increasing sequence. 
Then we get 
\begin{eqnarray*}
\int_{\R^N}\frac{1}{|x|^2}|\vp_m(\cdot-x_n)|^2\,dx&=&
\int_{\R^N}\frac{1}{|x+x_n|^2}|\vp_m|^2\,dx=
\int_{B(0,R_m)}\frac{1}{|x+x_n|^2}|\vp_m|^2\,dx\\
&\leq& \frac{1}{(|x_n|-R_m)^2}\int_{B(0,R_m)}|\vp_m|^2\,dx
\leq \frac{1}{m^2} |\vp_m|_2^2\\
&\to& 0.
\end{eqnarray*}
In view of the Hardy inequality \eqref{Hardyineq1} we conclude.
\end{proof}

\begin{Lem}\label{Lem_PS}
If $(u_n)\subset X$ and $(\mu_n)\subset [0,+\infty)$ are such that $\mu_n\leq \mu<\frac{(N-2)^2}{4}\mu(V)^+$, $(1+\|u_n\|)\J'_{\mu_n}(u_n)\to 0$ and $\J_{\mu_n}(u_n)$ is bounded from above, then $(u_n)$ is bounded. In particular any Cerami sequence of $\J_{\mu}$ is bounded for $0\leq \mu<\frac{(N-2)^2}{4}\mu(V)^+$.
\end{Lem}

\begin{proof}
Suppose that $(1+\|u_n\|)\J'_{\mu_n}(u_n)\to 0$, $\J_{\mu_n}(u_n)\leq M$ and $\|u_n\|\to\infty$ as $n\to\infty$. Let $v_n:=\frac{u_n}{\|u_n\|}$. We may assume that $v_n\rightharpoonup v$ in $X$ and $v_n(x)\to v(x)$ a.e. in $\R^N$. Moreover there is a sequence $(y_n)_{n\in\N}\subset\R^N$ such that
\begin{equation}\label{EqLemLions}
\liminf_{n\to\infty}\int_{B(y_n,1)}|v_n^+|^2\,dx >0.
\end{equation}
Otherwise, in view of Lions lemma (see \cite{Willem}[Lemma 1.21]) we get that $v_n^+\to 0$ in $L^t(\R^N)$ for $2<t<2^*$. 
By Lemma \ref{eqEstOfF} we obtain
$\int_{\R^N}F(x,sv_n^+)\,dx\to 0$ for any $s>0$. Let us fix any $s>0$. Observe that Lemma \ref{Lem_IneqForJ} implies that
$$\J_{\mu_n}(u_n)\geq \J_{\mu_n}(sv_n^+)+o(1)$$
and by \eqref{eqNormMuEst}
\begin{eqnarray}\label{EqIneq2}
M&\geq& \limsup_{n\to\infty}\J_{\mu_n}(u_n)\geq \limsup_{n\to\infty} \J_{\mu_n}(sv_n^+)\ge \frac{s^2}{2}\limsup_{n\to\infty}\|v_n^+\|^2_{\mu_n}\\
&\geq &\frac{s^2}{4}\Big(\mu(V)^+ -\frac{4\mu}{(N-2)^2}\Big) 
\limsup_{n\to\infty}\|v_n^+\|^2.\nonumber
\end{eqnarray}
Note that by \eqref{ARfor2}
$$\|u_n^+\|^2-\|u_n^-\|^2\geq\J_{\mu_n}'(u_n)(u_n).$$
Hence 
\begin{eqnarray*}
2\|u_n^+\|^2&\geq& \|u_n^+\|^2+\|u_n^-\|^2+\J_{\mu_n}'(u_n)(u_n)=\|u_n\|^2+\J_{\mu_n}'(u_n)(u_n)
\end{eqnarray*}
and, passing to a subsequence if necessary, $C:=\limsup_{n\to\infty}\|v_n^+\|^2>0$. Then
 by (\ref{EqIneq2})
$$M\geq \frac{s^2}{4}C\Big(\mu(V)^+ -\frac{4\mu}{(N-2)^2}\Big)$$
for any $s\geq0$ and the obtained contradiction shows that (\ref{EqLemLions}) holds. We may assume that $(y_n)_{n\in\N}\subset\Z^N$ and
$$\liminf_{n\to\infty}\int_{B(y_n,\rho)}|v_n^+|^2\,dx >0$$
for some $\rho>1$. Therefore passing to a subsequence $v_n^+(\cdot+y_n)\to v^+$ in $L^2_{loc}(\R^N)$ and $v^+\neq 0$. Note that if $v(x)\neq 0$ then $|u_n(x+y_n)|=|v_n(x+y_n)|\|u_n\|\to\infty$
and by (F4)
$$\frac{F(x,u_n(x+y_n))}{\|u_n\|^2}=\frac{F(x,u_n(x+y_n))}{|u_n(x+y_n)|^2}|v_n(x+y_n)|^2\to\infty$$
as $n\to\infty$.
Since $\J_{\mu_n}'(u_n)(u_n)\to 0$, then 
$$\|u_n^+\|^2-\|u_n^-\|^2 -\int_{\R^N}\frac{\mu_n}{|x|^2}|u_n|^2\,dx+o(1)=\int_{\R^N}f(x,u_n)u_n\,dx\geq 0$$
and $\frac{1}{\|u_n\|^2}\int_{\R^N}\frac{\mu_n}{|x|^2}|u_n|^2\,dx$ is bounded.
Therefore by the $\Z^N$-periodicity of $F$ in $x\in\R^N$ and by Fatou's lemma
\begin{eqnarray*}
0=\limsup_{n\to\infty}\frac{\J_{\mu_n}(u_n)}{\|u_n\|^2}&=&\limsup_{n\to\infty}\Big(\frac{1}{2}\big(\|v_n^+\|^2-\|v_n^-\|^2-\frac{1}{\|u_n\|^2}\int_{\R^N}\frac{\mu_n}{|x|^2}|u_n|^2\,dx\big)\\
&&-\int_{\R^N}\frac{F(x,u_n(x+y_n))}{\|u_n\|^2}\,dx\Big)\\
&=&-\infty.
\end{eqnarray*}
Thus we get a contradiction.
\end{proof}

Below we provide a decomposition of bounded Palais-Smale sequences of $\J_{\mu}$.
\begin{Lem}\label{LemStruwe}
Assume that $0\le\mu<\overline{\mu}$
and let $(u_n)$ be a bounded Palais-Smale sequence of $\J_{\mu}$ at level $c\ge0$, that is $\J_{\mu}'(u_n)\to 0$ and $\J_\mu(u_n)\to c$.
Then there is $k\geq 0$ and there are sequences
$(\bar{u}_i)_{i=0}^k\subset X$ and $(x_n^i)_{0\leq i\leq k}\subset \Z^N$ such that $x_n^0=0$,
$|x_n^i|\rightarrow +\infty, |x_n^i-x_n^j|\rightarrow +\infty, i\neq j, i,j=1,2,\cdots,k$, and passing to a subsequence, the following conditions hold:
\begin{eqnarray}
&& \J'_{\mu}(\bar{u}_0)=0,\nonumber\\
&& \J'_0(\bar{u}_i)=0\text{ and }\bar{u}_i\neq 0\text{ for } i=1,...,k,\nonumber\\
&& u_n-\sum_{i=0}^k \bar{u}_i(\cdot -x_n^i)\to0 \hbox{ in } X\hbox{ and }\|u_n\|^2\to\sum_{i=0}^k \|\bar{u}_i\|^2
\text{ as }n\to\infty,\nonumber\\
&& c=\J_\mu(\bar{u}_0)+\sum_{i=1}^k
\J_0(\bar{u}_i).\label{EqSplitLast}
\end{eqnarray}
\end{Lem}

Since proof of Lemma \ref{LemStruwe} is technical we postpone it to Section \ref{sec:Decomposition}.

\begin{altproof}{Theorem \ref{ThMain1}}
Applying Theorem \ref{ThAbstract} we find  a Cerami sequence $(u_n)$ of $\J_{\mu}$ at level $c_{\mu}>0$ and in view Lemma \ref{Lem_PS}, $(u_n)$ is bounded in $X$. If $\mu=0$, then by Theorem \ref{ThAbstract} we obtain
$$\inf_{\cN_{0}}\J_0\geq c_0 >0$$
and Lemma \ref{LemStruwe} implies that
there is a nontrivial critical point  $u_0\in\cN_0$ of $\J_0$ such that $\J_0(u_0)=c_0$. Hence $u_0$ is a ground state of $\J_0$, that is
 $\J_0(u_0)=\inf_{\cN_0}\J_0$.
Now let us assume that $0<\mu<\frac{(N-2)^2}{4}\mu(V)^+$ and consider
$$
M(u_0):=\{u=t u_0+v \in X |\; v \in X^-,\;
\|u\|\leq R(u_0),\; t\geq0 \},
$$
where $R(u_0)$ is given by Lemma \ref{LemGeometry}.
Observe that, if $t_nu_0+v_n\weakto t_0u_0+v_0$ in $X$ and $t_nu_0+v_n\in M(u_0)$ for $n\geq 1$, then passing to a subsequence we may assume that $t_n\to t_0$, $v_n\weakto v_0$ in $L^2(\R^N,\frac{1}{|x|^2})$, $v_n(x)\to v_0(x)$ a.e. on $\R^N$. Hence $t_0u_0+v_0\in M(u_0)$ and by Fatou's lemma $$\limsup_{n\to\infty}\J_{\mu}(t_nu_0+v_n)\leq \J_{\mu}(t_0u_0+v_0).$$ 
Therefore $M(u_0)$ is weakly sequentially closed and $\J_{\mu}$ is weakly sequentially upper semicontinuous. Then $\J_{\mu}$ attains its maximum in $M(u_0)$, that is, there is
$t_0u_0+v_0\in M(u_0)$ such that
$$\J_{\mu}(t_0u_0+v_0)\geq \J_{\mu}(u)$$
for any $u\in M(u_0)$. Note that by Lemma \ref{LemGeometry}, $\J_{\mu}(t_0u_0+v_0)>0$ and hence $t_0u_0+v_0\neq 0$.
Define $h(s,u)=u$ for $s\in [0,1]$ and $u\in M(u_0)$ and note that (h1)-(h4) are satisfied, that is $h\in\Gamma(u_0)$.
Then in view of Lemma \ref{Lem_IneqForJ} we have
\begin{equation}
c_0=\J_0(u_0)\geq \J_0(t_0u_0+v_0) > \J_\mu(t_0u_0+v_0)=\max_{u\in M(u_0)}\J_\mu(h(1,u))\geq c_\mu.
\end{equation}
Thus by (\ref{EqSplitLast}) we get $k=0$ and $\J_\mu(\bar{u}_0)=c_\mu>0$. Therefore $\bar{u}_0$ is a nontrivial critical point of $\J_\mu$ and by Theorem \ref{ThAbstract} 
$$c_\mu=\inf_{\cN_\mu}\J_\mu.$$
\end{altproof}

\section{Proof of Theorem \ref{ThMain2}}\label{sec:ProofTh2}

\begin{Lem}\label{lem:elementsOnN}
If $0\leq \mu <\frac{(N-2)^2}{4}\mu(V)^+$ then for every $u\in X\setminus X^-$ there is $t>0$ and $v\in X^-$ such that $tu+v\in\cN_{\mu}$. If $X^-=\{0\}$ and $\mu<\bar \mu$ then for every $u\in X\setminus \{0\}$ there is $t>0$ such that $tu\in\cN_{\mu}$.
\end{Lem}
\begin{proof}
Let $u\in X\setminus X^-$ and consider a map $\xi:\R^+\times X^-\to\R$ such that 
$$\xi(t,v)=-\J_{\mu}(tu^++v).$$
Similarly as in proof of Theorem \ref{ThMain1} we show that $\xi$ is weakly lower semicontinuous for $\mu\geq 0$. Since $\xi$ is bounded from below and coercive (see proof of Lemma \ref{LemGeometry}), then we find
$t\geq 0$ and $v\in X^-$ such that 
$$\J_{\mu}(tu+v)=\sup_{\R^+u^+\oplus X^-}\J_\mu,$$
where $\R^+u^+:=\{tu^+|\;t\geq 0\}$.
In view of Lemma \ref{LemGeometry}  and Lemma \ref{Lem_IneqForJ} (condition (A5)) we get $t>0$, hence
$t u+v\in\mathcal{N}_\mu$. Observe that if $X^-=\{0\}$, then $\xi$ is continuous, coercive and bounded from below for any $\mu\in\R$. Since the first inequality in (\ref{eqLinkingGeo}) actually holds for any $\mu<\bar \mu$, therefore there is $t>0$ such that $tu\in\cN_{\mu}$.
\end{proof}

\begin{altproof}{Theorem \ref{ThMain2}}\\
(a) Let $u_{\mu}\in\cN_\mu$ be a ground state of $\J_{\mu}$ and $-\frac{4}{(N-2)^2}\mu(V)^-<\mu <\frac{(N-2)^2}{4}\mu(V)^+$. 
In view of Lemma \ref{lem:elementsOnN} 
there is $t> 0$ and $v\in X^-$ such that 
$t u_{\mu}+v\in\mathcal{N}_0$.
Then by Lemma \ref{Lem_IneqForJ}
\begin{eqnarray}\label{Th2ineq1}
c_{\mu}&=&\J_{\mu}(u_{\mu})\geq \J_{\mu}(tu_{\mu}+v)=\J_{0}(tu_{\mu}+v)-\frac{1}{2}\int_{\R^N}\frac{\mu}{|x|^2}|tu_{\mu}+v|^2\,dx\\\nonumber
&\geq& c_0-\frac{1}{2}\int_{\R^N}\frac{\mu}{|x|^2}|tu_{\mu}+v|^2\,dx
\end{eqnarray}
and we obtain \eqref{ConditonOnmu2}.
Now, let us assume that $u_0\in\cN_0$ is a ground of $\J_{0}$. Similarly as above we show \eqref{ConditonOnmu1}, that is,   
\begin{eqnarray*}
c_{0}&=&\J_{0}(u_0)\geq \J_{0}(t'u_0+v')=\J_{\mu}(t'u_0+v')+\frac{1}{2}\int_{\R^N}\frac{\mu}{|x|^2}|t'u_0+v'|^2\,dx\\\nonumber
&\geq& c_\mu+\frac{1}{2}\int_{\R^N}\frac{\mu}{|x|^2}|t'u_0+v'|^2\,dx
\end{eqnarray*}
for any $0\leq\mu\leq \frac{(N-2)^2}{4}\mu(V)^+$ and some $t'> 0$ and $v'\in X^-$ such that 
$t'u+v'\in\mathcal{N}_\mu$. Observe that we get
$$c_0\geq c_\mu=\J_{\mu}(u_\mu)$$
and by Lemma \ref{Lem_PS} we have that $(u_\mu)$ is bounded if $\mu\to 0^+$. Take any sequence $\mu_n\to 0^+$ such that $\mu_n\leq\mu<\frac{4}{(N-2)^2}\mu(V)^+$ and denote $u_n:=u_{\mu_n}$. 
Then there is a sequence $(y_n)_{n\in\N}\subset\R^N$ such that
\begin{equation}\label{EqLemLions2}
\liminf_{n\to\infty}\int_{B(y_n,1)}|u_n^+|^2\,dx >0.
\end{equation}
Otherwise, in view of Lions lemma we get that $u_n^+\to 0$ in $L^t(\R^N)$ for $2<t<2^*$ and since $u_n\in \cN_{\mu_n}$, then
$$\|u_n^+\|^2=\int_{\R^N}\frac{\mu_n}{|x|^2}u_nu_n^+\,dx+\int_{\R^N}f(x,u_n)u_n^+\,dx\to 0$$
as $n\to\infty$. 
Hence $\limsup_{n\to\infty}J_{\mu_n}(u_n)\leq 0$, but this contradicts the following inequalities 
$$\J_{\mu_n}(u_n)\geq \J_{\mu_n}\Big(\frac{r}{\|u_n^+\|}u_n^+\Big)\geq  
\inf_{n\in \N}\inf_{u\in X^+,\;\|u\|=r}\J_{\mu_n}(u)>0,$$
for sufficiently small $r>0$, where the last inequality follows from similar arguments provided in Lemma \ref{LemGeometry}. Therefore \eqref{EqLemLions2} holds and we may assume that $(y_n)_{n\in\N}\subset\Z^N$ and
$$\liminf_{n\to\infty}\int_{B(y_n,\rho)}|u_n^+|^2\,dx >0$$
for some $\rho>1$. Therefore passing to a subsequence we find $u\in X$ such that $u_n^+(\cdot+y_n)\to u^+$ in $L^2_{loc}(\R^N)$ and $u^+\neq 0$. Moreover we may assume that $$u_n(\cdot+y_n)\weakto u\hbox{ in }X\hbox{ and }u_n(x+y_n)\to u(x),\;u_n^+(x+y_n)\to u^+(x)\hbox{ a.e. on }\R^N.$$ 
Let $t_nu_n+v_n\in\cN_0$ and $t_n>0$, $v_n\in X^-$. Then by \eqref{ARfor2}
\begin{eqnarray}\label{ineqTh2}
\|u_n^+\|^2&=&\|u_n^-+v_n/t_n\|^2+\frac{1}{t_n^2}\int_{\R^N}f(x,t_nu_n+v_n)(t_nu_n+v_n)\,dx\\\nonumber
&\geq&\|u_n^-+v_n/t_n\|^2+2\int_{\R^N}\frac{F(x,t_n(u_n+v_n/t_n))}{t_n^2}\,dx. 
\end{eqnarray}
Therefore $\|u_n^-+v_n/t_n\|$ is bounded and we may assume that $u_n^-(x)+v_n(x)/t_n\to v(x)$ a.e. on $\R^N$ for some $v\in X^-$. Hence, if $t_n\to\infty$, then   
$|t_nu_n(x)+v_n(x)|=t_n|u_n(x)+v_n(x)/t_n|\to \infty$ provided that $u^+(x)+v(x)\neq 0$. In view of Fatou's lemma and by (F4)
$$\int_{\R^N}\frac{F(x,t_n(u_n+v_n/t_n))}{t_n^2}\,dx\to\infty,$$
which contradicts \eqref{ineqTh2}. Therefore $t_n$ is bounded, thus  
$\|t_nu_n^+\|$ and $\|t_nu_n^-+v_n\|$ are bounded and
by the Hardy inequality
$$\frac{1}{2}\int_{\R^N}\frac{\mu_n}{|x|^2}|t_nu_{n}+v_n|^2\,dx\to 0$$
as $n\to\infty$. Therefore by \eqref{Th2ineq1}
we get \eqref{ThMain2conv}.\\
(b) Let $(u_n)$ be a sequence of ground states of $\J_{\mu_n}$ as in (a) and take $x_n:=y_n$.  For $(x,u)\in\R^N\times\R$
$$G(x,u):=\frac{1}{2}f(x,u)u-F(x,u).$$
Observe that for any $\phi\in X$
\begin{eqnarray*}
\J_{0}'(u_n(\cdot+x_n))(\phi)
&=&\J_{\mu_n}'(u_n)(\phi(\cdot-x_n))+\int_{\R^N}\frac{\mu_n}{|x|^2}u_n\phi(\cdot-x_n)\, dx\\
&\to& 0 
\end{eqnarray*}
as $n\to\infty$. In view of the Vitali convergence theorem $\J_{0}'(u_n(\cdot+x_n))(\phi)\to \J_{0}'(u)(\phi)$. Thus $u$ is a nontrivial critical point of $\J_0$. Observe that by \eqref{ThMain2conv} and Fatou's lemma 
\begin{eqnarray}\label{EqConvG}
c_0&=&\liminf_{n\to\infty}\J_{\mu_n}(u_n)=\liminf_{n\to\infty}\Big(\J_{\mu_n}(u_n)-\frac{1}{2}\J_{\mu_n}'(u_n)(u_n)\Big)\\\nonumber
&=&\liminf_{n\to\infty}\int_{\R^N}G(x,u_n)\,dx
=\liminf_{n\to\infty}\int_{\R^N}G(x,u_n(x+x_n))\,dx
\geq \int_{\R^N}G(x,u)\,dx\\\nonumber
&=&\J_0(u)\geq c_0.
\end{eqnarray}
Thus we obtain that $u$ is a ground state of $\J_0$. Let us denote $w_n:=u_n(\cdot+x_n)$ and
observe that
\begin{eqnarray}\label{EqBrezisLieb}
\int_{\R^N}G(x,w_n)-G(x,w_n-u)\, dx
&=&\int_{\R^N}\int_0^1\frac{d}{dt}G(x,w_n-u+tu)\, dtdx\\\nonumber
&=&\int_0^1\int_{\R^N}g(x,w_n-u+tu)u\, dxdt,
\end{eqnarray}
where $g(x,u):=\partial_u G(x,u)$ for $u\in\R$ and a.e. $x\in\R^N$.
Since $(w_n-u+tu)$ is bounded in $X$, then
for any $\eps>0$ there is $\delta>0$
such that for any $\Omega$ with the Lebesgue measure $|\Omega|<\delta$, we have
$$\int_{\Omega}|g(x,w_n-u+tu)u|\, dx < \eps$$
for any $n\in\N$.
Thus $(g(x,w_n-u+tu)u)$ is uniformly integrable.
Moreover for any $\eps>0$ there is $\Omega\subset\R^N$,
$|\Omega|<+\infty$,
such that for any $n\in\N$
$$\int_{\R^N\setminus\Omega}|g(x,w_n-u+tu)u|\, dx < \eps.$$
Hence a family $(g(x,w_n-u+tu)u)$ is tight over $\R^N$.
Since $g(w_n-u+tu)u\to g(tu)u$ a.e. in $\R^N$, then
in view of the Vitali convergence theorem $g(x,tu)u$ is integrable and
$$\int_{\R^N}g(x,w_n-u+tu)u\, dx\to \int_{\R^N}g(x,tu)u\, dx$$
as $n\to\infty$.
By (\ref{EqBrezisLieb}) we obtain
$$\int_{\R^N}G(x,w_n)-G(x,w_n-u)\, dx\to
\int_0^1\int_{\R^N}g(x,tu)u\, dxdt=\int_{\R^N}G(x,u)\, dx$$
as $n\to\infty$.
Taking into account (\ref{EqConvG}) we get
$$\lim_{n\to\infty}\int_{\R^N}G(x,w_n-u_0)\, dx=0,$$
and by (F6) we have $w_n\to u$ in $L^q(\R^N)$. Since $(w_n)$ is bounded in $L^2(\R^N)$ and in $L^{2^*}(\R^N)$, then by the interpolation inequalities $w_n\to u$ in $L^t(\R^N)$ for $2<t<2^*$. Thus
\begin{eqnarray*}
\|w_n^+-u^+\|^2&=&\J_{\mu_n}'(u_n)((w_n^+-u^+)(\cdot-x_n))-\langle u^+,w_n^+-u^+\rangle\\
&&+\int_{\R^N}\frac{\mu_n}{|x|^2}u_n(w_n^+-u^+)(x-x_n)\,dx
+\int_{\R^N}f(x,w_n)(w_n^+-u^+)\,dx\\
&\to& 0
\end{eqnarray*}
and
\begin{eqnarray*}
\|w_n^--u^-\|^2&=&-\J_{\mu_n}'(u_n)((w_n^--u^-)(\cdot-x_n))-\langle u^-,w_n^--u^-\rangle\\
&&-\int_{\R^N}\frac{\mu_n}{|x|^2}u_n(w_n^--u^-)(x-x_n)\,dx
-\int_{\R^N}f(x,w_n)(w_n^--u^-)\,dx\\
&\to& 0
\end{eqnarray*}
as $n\to\infty$. Therefore $w_n\to u$ in $X$.\\
(c) Suppose that $\mu<0<\inf\sigma(-\Delta+V)$ and $u_\mu$ is a ground state of $\J_{\mu}$ and $u_0$ is a ground state of $\J_{0}$. Then $X^-=\{0\}$, $\mu(V)^-=\infty$ and by \eqref{ConditonOnmu2} we have 
\begin{equation}\label{ineqTh2c0mu}
c_0<c_{\mu}, 
\end{equation}
since
$t u_{\mu}+v\in\mathcal{N}_0$ and $t u_{\mu}+v\neq 0$.
In view of Lemma \ref{lem:elementsOnN}, for any $y\in\Z^N$, we find  $t_y>0$ such that $t_y u_0(\cdot+y)\in\mathcal{N}_\mu$. Then by \eqref{ARfor2} and (F4)
\begin{eqnarray}\label{eqOnN}
\|u_0\|^2-\int_{\R^N}\frac{\mu}{|x|^2}|u_0(\cdot+y)|^2\, dx
&=&\frac{1}{t_y^2}\int_{\R^N}f(x,t_y u_0)t_yu_0\, dx\\\nonumber
&\geq& 2\int_{\R^N}\frac{F(x,t_y u_0)}{t_y^2}\, dx, 
\end{eqnarray}
where the last integral tends to $\infty$ as $t_y\to\infty$. Therefore $(t_{y_n})$ is bounded if $|y_n|\to\infty$ and we may assume that $t_{y_n}\to t_0\geq 0$.
Observe that by Lemma \ref{LemIntegralConvergence}
$$\liminf_{n\to\infty}\Big(\frac{1}{2}t_{y_n}^2\|u_0\|^2\Big)\geq \liminf_{n\to\infty}\J_{\mu}(t_{y_n}u_0(\cdot+y_n))\geq c_{\mu}>0,$$
hence $t_0>0$.
Again by Lemma \ref{LemIntegralConvergence} 
we get
\begin{eqnarray*}
\J_{0}'(t_{y_n}u_0)(u_0)&=&\J_{0}'(t_{y_n}u_0(\cdot+y_n))(u_0(\cdot+y_n))\\
&=&\J_{\mu}'(t_{y_n}u_0(\cdot+y_n))(u_0(\cdot+y_n))
+\int_{\R^N}\frac{\mu}{|x|^2}|u_0(x+y_n)|^2t_{y_n}\,dx\\
&\to& 0. 
\end{eqnarray*}
Thus $\J_{0}'(t_0u_0)(u_0)=0$ and $t_0u_0\in\cN_0$. Note that
by Lemma \ref{LemIntegralConvergence} and Lemma \ref{Lem_IneqForJ} (condition (A5)) we obtain
\begin{eqnarray*}
c_\mu&\leq& \lim_{n\to\infty}\J_\mu(t_{y_n}u_0(\cdot+y_n))\\
&=& \J_0(t_{0}u_0) = \J_0(u_0) = c_{0},
\end{eqnarray*}
which contradicts \eqref{ineqTh2c0mu}.
\end{altproof}

\section{Decomposition of bounded Palais-Smale sequences}\label{sec:Decomposition}

In this section we obtain a variant of \cite{JeanjeanTanakaIndiana2005}[Theorem 5.1] for bounded Palais-Smale sequences, hence also for Cerami sequences, of strongly indefinite functionals involving sum of periodic and inverse square potentials. 

\begin{altproof}{Lemma \ref{LemStruwe}}
Let $(u_n)$ be a bounded Palais-Smale sequence of $\J_{\mu}$ at $c\ge0$. Then $(u_n)$ is bounded in $L^2\big(\mathbb{R}^N,\frac{1}{|x|^2}\big)$ by (\ref{Hardyineq1}) and we may assume that
\begin{eqnarray*}
u_n&\rightharpoonup& \overline{u}_0~ \text{in}~X,\\
u_n&\rightharpoonup& \overline{u}_0~ \text{in}~L^2\Big(\mathbb{R}^N,\frac{1}{|x|^2}\Big),\\
u_n&\rightarrow& \overline{u}_0~ \text{in}~L_{loc}^2(\mathbb{R}^N),\\
u_n&\rightarrow& \overline{u}_0~ a.e.~ \text{on}~\mathbb{R}^N.
\end{eqnarray*} 
Then $\J'_{\mu}(\overline{u}_0)=0$.
Denote $v_n=u_n-\overline{u}_0$. Then $v_n^+\rightharpoonup 0$ in $X^+,$ $v_n^-\rightharpoonup 0$ in $X^-,$ and $v_n\rightharpoonup 0$ in $L^2\big(\mathbb{R}^N,\frac{1}{|x|^2}\big)$, hence
\begin{eqnarray}\label{Eq-PS1-LemStruwe}
\|v_n^+\|^2&=&\|u_n^+\|^2-\|\overline{u}_0^+\|^2+o(1),\;\|v_n^-\|^2=\|u_n^-\|^2-\|\overline{u}_0^-\|^2+o(1),\\\label{Eq-PS1-LemStruwe-add1}
\|v_n\|^2&=&\|u_n\|^2-\|\overline{u}_0\|^2+o(1),
\\\label{Eq-PS2-LemStruwe}
\int_{\mathbb{R}^N}\frac{v_n^2}{|x|^2}&=&\int_{\mathbb{R}^N}\frac{u_n^2}{|x|^2}-\int_{\mathbb{R}^N}\frac{\overline{u}_0^2}{|x|^2}+o(1).
\end{eqnarray}
Now we prove that 
\begin{eqnarray}\label{Eq-PS3-LemStruwe}
&&\int_{\mathbb{R}^N}F(x,v_n)=\int_{\mathbb{R}^N}F(x,u_n)-\int_{\mathbb{R}^N}F(x,\overline{u}_0)+o(1).
\end{eqnarray}
In fact, by (F2), (F3), Vitali's theorem implies 
\begin{eqnarray*}
&&\int_{\mathbb{R}^N}F(x,u_n)-F(x,u_n-\overline{u}_0)\\
&=&-\int_{\mathbb{R}^N}\int_0^1 \frac{d}{d\theta}F(x,u_n-\theta\overline{u}_0)=\int_{\mathbb{R}^N}\int_0^1 f(x,u_n-\theta\overline{u}_0)\overline{u}_0\\
&\rightarrow& \int_{\mathbb{R}^N}\int_0^1 f(x,\overline{u}_0-\theta\overline{u}_0)\overline{u}_0=\int_{\mathbb{R}^N}\int_0^1 \frac{d}{d\theta}F(x,\overline{u}_0-\theta\overline{u}_0)=\int_{\mathbb{R}^N}F(x,\overline{u}_0).
\end{eqnarray*}
Therefore (\ref{Eq-PS1-LemStruwe}), (\ref{Eq-PS2-LemStruwe}) and (\ref{Eq-PS3-LemStruwe}) give that 
\begin{eqnarray}\label{Eq-PS4-LemStruwe}
\J_{\mu}(v_n)&=&\J_{\mu}(u_n)-\J_{\mu}(\overline{u}_0)+o(1).
\end{eqnarray} 

Now we distinguish two cases.

{\em Case 1.}  $\lim\limits_{n\rightarrow\infty}\sup\limits_{y\in\mathbb{R}^N}\int_{B(y,1)}|v_n|^2\,dx=0.$\\
In view of Lion's lemma, we have $v_n\to 0$ in $L^t(\mathbb{R}^N)$ for $2<t<2^*$. Since the orthogonal projection of $X$ on $X^+$
is continuous in the $L^t$-norm \cite{Troestler}, then
$v_n^+\rightarrow0$ and  $v_n^-\rightarrow0$ in $L^t(\mathbb{R}^N)$ for $2<t<2^*$.
Moreover using $\J'_{\mu}(u_n)=o(1)$ and $\J'_{\mu}(\overline{u}_0)=0$ we obtain
\begin{eqnarray*}
o(1)&=&\J'_{\mu}(u_n)v_n^+\\
&=&\int_{\mathbb{R}^N}\nabla u_n\nabla v_n^++V(x)u_n v_n^+\,dx-\mu\int_{\mathbb{R}^N} \frac{u_n v_n^+}{|x|^2}\,dx-\int_{\mathbb{R}^N}f(x,u_n)v_n^+\,dx\\
&=&\|v_n^+\|^2-\mu\int_{\mathbb{R}^N} \frac{v_nv_n^+}{|x|^2}+\J'_{\mu}(\overline{u}_0)v_n^++\int_{\mathbb{R}^N}f(x,\overline{u}_0)v_n^+-\int_{\mathbb{R}^N}f(x,u_n)v_n^+\\
&=&\|v_n^+\|^2-\mu\int_{\mathbb{R}^N} \frac{v_nv_n^+}{|x|^2}+\int_{\mathbb{R}^N}f(x,\overline{u}_0)v_n^+-\int_{\mathbb{R}^N}f(x,u_n)v_n^+,
\end{eqnarray*}
which combining with (\ref{Hardyineq1}) implies that
\begin{equation}\label{EqLemStruwe1}
\Big(1-\frac{4\mu}{(N-2)^2}\Big)\|v_n^+\|^2\leq\mu\int_{\mathbb{R}^N} \frac{v_n^+v_n^-}{|x|^2}+\int_{\mathbb{R}^N}f(x,u_n)v_n^+-\int_{\mathbb{R}^N}f(x,\overline{u}_0)v_n^++o(1).
\end{equation}
Similarly, we have
\begin{eqnarray*}
o(1)&=&\J'_{\mu}(u_n)v_n^-\\
&=&-\|v_n^-\|^2-\mu\int_{\mathbb{R}^N} \frac{v_nv_n^-}{|x|^2}\,dx+\int_{\mathbb{R}^N}f(x,\overline{u}_0)v_n^-\,dx-\int_{\mathbb{R}^N}f(x,u_n)v_n^-\,dx
\end{eqnarray*}
and
\begin{equation}\label{EqLemStruwe1addadd1}
\|v_n^-\|^2\leq-\mu\int_{\mathbb{R}^N} \frac{v_n^+v_n^-}{|x|^2}\,dx+\int_{\mathbb{R}^N}f(x,\overline{u}_0)v_n^-\,dx-\int_{\mathbb{R}^N}f(x,u_n)v_n^-\,dx+o(1).
\end{equation}
Hence
\begin{eqnarray}\label{EqLemStruwe1addadd2}
\Big(1-\frac{4\mu}{(N-2)^2}\Big)\|v_n\|^2
&\leq&\int_{\mathbb{R}^N}f(x,\overline{u}_0)v_n^--\int_{\mathbb{R}^N}f(x,u_n)v_n^-\,dx\\
&&+\int_{\mathbb{R}^N}f(x,u_n)v_n^+\,dx-\int_{\mathbb{R}^N}f(x,\overline{u}_0)v_n^+\,dx+o(1).\nonumber
\end{eqnarray}
Since
$v_n^+\rightarrow0$ and  $v_n^-\rightarrow0$ in $L^t(\mathbb{R}^N)$ for $2<t<2^*$, then $v_n\to 0$ in $X$  
and we end the proof with $k=0$.

{\em Case 2.} There is a sequence $(y_n)\subset\Z^N$ such that
\begin{equation}\label{Case2LemStruwe1}
\liminf\limits_{n\rightarrow\infty}\int_{B(y_n,\rho)}|v_n|^2\,dx>0
\end{equation}
for some $\rho>1$.
Passing to a subsequence we may assume that $|y_n|\rightarrow+\infty$. 
Let $\widehat{u}_n=u_n(x+y_n)$ and note that by (\ref{Case2LemStruwe1})
we find $\overline{u}_1\neq 0$ such that up to a subsequence
\begin{eqnarray*}
\widehat{u}_n&\rightharpoonup& \overline{u}_1~ \text{in}~X,\\
\widehat{u}_n&\rightarrow& \overline{u}_1~ \text{in}~L_{loc}^2(\mathbb{R}^N),\\
\widehat{u}_n&\rightarrow& \overline{u}_1~ a.e.~ \text{on}~\mathbb{R}^N.
\end{eqnarray*}
In view of the H\"older inequality and Lemma \ref{LemIntegralConvergence} 
we get for any $\phi\in X$
$$\int_{\R^N}\frac{1}{|x|^2}u_n(x)\phi(x-y_n)\,dx\to 0\hbox{ as }n\to\infty.$$
Then
\begin{eqnarray*}
o(1)&=&\J_\mu'(u_n)(\phi(x-y_n))\\
&=&\int_{\mathbb{R}^N}\nabla u_n\nabla\phi(x-y_n)+V(x)u_n\phi(x-y_n)\,dx-\int_{\R^N}\frac{\mu}{|x|^2}u_n\phi(x-y_n)\,dx\\
&&-\int_{\mathbb{R}^N} f(x,u_n)\phi(x-y_n)\,dx\\
&=&\int_{\mathbb{R}^N}\nabla\widehat{u}_n\nabla\phi+V(x)\widehat{u}_n\phi\,dx-\int_{\mathbb{R}^N} f(x,\widehat{u}_n)\phi\,dx+o(1)\\
&=&\J'_0(\widehat{u}_n)(\phi)+o(1)
\end{eqnarray*}
which implies that $\J'_0(\overline{u}_1)(\phi)=0$ and $\overline{u}_1$ is a critical point of $\J_0$.
Now denote 
$$z_n:=u_n-\overline{u}_0-\overline{u}_1(\cdot-y_n)$$ 
and since $\overline{u}_1(\cdot-y_n)\rightharpoonup 0$ in $X$
then
\begin{eqnarray}\label{Energyfunc-Zn1}
&&\|z_n^+\|^2=\|u_n^+\|^2-\|\overline{u}_0^+\|^2-\|\overline{u}_1^+\|^2+o(1),~
\|z_n^-\|^2=\|u_n^-\|^2-\|\overline{u}_0^-\|^2-\|\overline{u}_1^-\|^2+o(1),\\\label{Energyfunc-Zn1-add1}
&&
\|z_n\|^2=\|u_n\|^2-\|\overline{u}_0\|^2-\|\overline{u}_1\|^2+o(1).
\end{eqnarray}
In view of Lemma \ref{LemIntegralConvergence}
\begin{eqnarray}\label{Energyfunc-Zn3}
\int_{\mathbb{R}^N}\frac{z_n^2}{|x|^2}\,dx&=&\int_{\mathbb{R}^N}\frac{(u_n-\overline{u}_0)^2}{|x|^2}\,dx-\int_{\mathbb{R}^N}\frac{\overline{u}_1^2(x-y_n)}{|x|^2}\,dx+o(1)\\\nonumber
&=&\int_{\mathbb{R}^N}\frac{u_n}{|x|^2}\,dx-\int_{\mathbb{R}^N}\frac{\overline{u}_0^2}{|x|^2}\,dx+o(1).
\end{eqnarray}
Observe that $z_n(x+y_n)\to \overline{u}_1(x)$ a.e. on $\R^N$ and by Vitali's theorem and (\ref{Eq-PS3-LemStruwe}) we get
\begin{eqnarray}\label{Energyfunc-Zn6}
\int_{\mathbb{R}^N}F(x,z_n)\,dx&=&\int_{\mathbb{R}^N}F(x,z_n(x+y_n))\,dx\\
&=&\int_{\mathbb{R}^N}F(x,u_n-\overline{u}_0)\,dx-\int_{\mathbb{R}^N}F(x,\overline{u}_1)\,dx+o(1)\nonumber\\
&=&\int_{\mathbb{R}^N}F(x,u_n)\,dx-\int_{\mathbb{R}^N}F(x,\overline{u}_0)\,dx-\int_{\mathbb{R}^N}F(x,\overline{u}_1)\,dx+o(1).\nonumber
\end{eqnarray}
Now by (\ref{Energyfunc-Zn1})-(\ref{Energyfunc-Zn6}), we have
\begin{equation}\label{Energyfunc-Zn7}
\J_{\mu}(z_n)=\J_{\mu}(u_n)-\J_{\mu}(\overline{u}_0)-\J_0(\overline{u}_1)+o(1)
\end{equation} 
and we take $x_n^1:=y_n$.
Now we replace $v_n$ by $z_n$ and repeat the above arguments in {\em Case 1} and {\em Case 2}, that is if
\begin{equation}
\label{eq:Case1}
\lim\limits_{n\rightarrow\infty}\sup\limits_{y\in\mathbb{R}^N}\int_{B(y,1)}|z_n|^2\,dx=0,
\end{equation}
then $z_n\to 0$ in $X$ and in view of \eqref{Energyfunc-Zn1} and \eqref{Energyfunc-Zn7} we take $k=1$. Otherwise as in {\em Case 2} we find $(y_n)\subset\Z^N$ such that
\eqref{Case2LemStruwe1} holds for $(z_n)$. Then passing to a subsequence $|y_n|\to\infty$ and $|y_n-x_n^1|\to\infty$ as $n\to\infty$.
Similarly as above $\widehat{u}_n=u_n(x+y_n)$ and
we find $\overline{u}_2\neq 0$ such that up to a subsequence
\begin{eqnarray*}
\widehat{u}_n&\rightharpoonup& \overline{u}_2~ \text{in}~X,\\
\widehat{u}_n&\rightarrow& \overline{u}_2~ \text{in}~L_{loc}^2(\mathbb{R}^N),\\
\widehat{u}_n&\rightarrow& \overline{u}_2~ a.e.~ \text{on}~\mathbb{R}^N,
\end{eqnarray*}
$\overline{u}_2$ is a critical point of $\J_0$.
Now denote new $z_n:=u_n-\overline{u}_0-\overline{u}_1(\cdot-x_n^1)-\overline{u}_2(\cdot-y_n)$ and 
similarly to 
(\ref{Energyfunc-Zn1}), (\ref{Energyfunc-Zn1-add1}) and (\ref{Energyfunc-Zn7}) we obtain
\begin{eqnarray*}
\|z_n^+\|^2&=&\|u_n^+\|^2-\|\overline{u}_0^+\|^2-\|\overline{u}_1^+\|^2-\|\overline{u}_2^+\|^2+o(1),\\
\|z_n^-\|^2&=&\|u_n^-\|^2-\|\overline{u}_0^-\|^2-\|\overline{u}_1^-\|^2-\|\overline{u}_2^-\|^2+o(1),\\
\|z_n\|^2&=&\|u_n\|^2-\|\overline{u}_0\|^2-\|\overline{u}_1\|^2-\|\overline{u}_2\|^2+o(1),\\
\J_{\mu}(\overline{z}_n)&=&\J_{\mu}(u_n)-\J_{\mu}(\overline{u}_0)-\J_0(\overline{u}_1)-\J_0(\overline{u}_2)+o(1),
\end{eqnarray*}
and $x_n^2:=y_n$. Again we repeat the above arguments in {\em Case 1} and {\em Case 2} and the iterations must stop after finite steps, since there is a constant $\rho_0>0$ such that
\begin{equation}\label{eqEstimJ_0Critical}
\|u\|\geq \rho_0\hbox{ for any }u\neq 0\hbox{ such that }\J_0'(u)=0.
\end{equation}
Indeed $\J'_0(u)(u^+)=0$, (F2) and (F3) imply that for any $\eps>0$ there is
a constant $C_1>0$ such that 
\begin{eqnarray*}
\|u^+\|^2\leq\int_{\R^N}|f(x,u)u^+|\,dx\leq \eps\|u^+\|\|u\|+C_1\|u^+\|\|u\|^{p-1}.
\end{eqnarray*}
Similarly by $\J'_0(u)(u^-)=0$ we get a constant $C_2>0$ such that
\begin{eqnarray*}
\|u^-\|^2\leq \int_{\R^N}|f(x,u)u^-|\,dx\leq \eps\|u^-\|\|u\|+C_2\|u^-\|\|u\|^{p-1}.
\end{eqnarray*}
Hence
\begin{eqnarray*}
\|u\|^2\leq 2\eps\|u\|^2+2\max\{C_1,C_2\}\|u\|^{p}
\end{eqnarray*}
and \eqref{eqEstimJ_0Critical} is satisfied, which completes the proof.
\end{altproof}

{\bf Acknowledgements.} The first author was supported by the National Natural Science Foundation of China (Grant Nos.11271299, 11001221) and the
Fundamental Research Funds for the Central Universities (Grant No. 3102015ZY069). The second author was partially supported by the NCN Grant no. 2014/15/D/ST1/03638 and he would like to thank the members of the Department of Applied Mathematics of the Northwestern Polytechnical University in Xi'an, where part of this work has been done, for their invitation and hospitality.

{\sc Address of the authors:}\\[1em]
\parbox{8.5cm}{Qianqiao Guo\\
 Department of Applied Mathematics\\
 Northwestern Polytechnical University,\\
 Postbox 894\\
 710129 Xi'an\\
 China\\
 gqianqiao@nwpu.edu.cn}
\parbox{11cm}{\vspace{5mm}Jaros\l aw Mederski\\
 Faculty of Mathematics and Computer Science\\
 Nicolaus Copernicus University \\
 ul.\ Chopina 12/18\\
 87-100 Toru\'n\\
 Poland\\
 jmederski@mat.umk.pl\\
 }

\end{document}